\documentclass[leqno]{amsart}
\usepackage{amsmath}
\usepackage{amssymb}
\usepackage{amsthm}
\usepackage{enumerate}
\usepackage[mathscr]{eucal}
\theoremstyle{plain}
\newtheorem{theorem}{Theorem}[section]

\newtheorem{cor}{Corollary}[theorem]

\theoremstyle{definition}

\newtheorem{remark}{Remark}[section]

\usepackage[pagewise]{lineno}
\begin{document}
\title[Approximate Birkhoff-James orthogonality and smoothness in $\mathbb{L}(\mathbb{X},\mathbb{Y})$]{Approximate Birkhoff-James orthogonality and smoothness in the space of bounded linear operators}
\author[Arpita Mal, Kallol Paul, T.S.S.R.K. Rao and  Debmalya Sain]{Arpita Mal, Kallol Paul, T.S.S.R.K. Rao and  Debmalya Sain}

\newcommand{\acr}{\newline\indent}
\address[Mal]{Department of Mathematics\\ Jadavpur University\\ Kolkata 700032\\ West Bengal\\ INDIA}
\email{arpitamalju@gmail.com}

\address[Paul]{Department of Mathematics\\ Jadavpur University\\ Kolkata 700032\\ West Bengal\\ INDIA}
\email{kalloldada@gmail.com}

\address[Rao] {Theoretical Statistics and Mathematics Unit\\ Indian Statistical Institute\\ Bangalore\\ India}
\email{tss@isibang.ac.in}

\address[Sain]{Department of Mathematics\\ Indian Institute of Science\\ Bengaluru 560012\\ Karnataka \\India\\ }
\email{saindebmalya@gmail.com}

\thanks{The first and second author acknowledges the generosity of Indian Statistical Institute, Bangalore,  and in particular, Professor T.S.S.R.K. Rao,   for supporting the visit to the Institute during June 2018. This research paper originated from that visit. First author would like to thank UGC, Govt. of India for the financial support. The research of Prof. Paul  is supported by project MATRICS(MTR/2017/000059)  of DST, Govt. of India. The research of Dr. Debmalya Sain is sponsored by Dr. D.S. Kothari Postdoctoral Fellowship under the mentorship of Prof. Gadadhar Misra. Dr. Sain feels elated to acknowledge the motivating presence of his younger brother Debdoot in every sphere of his life. } 


\begin{abstract}
We study approximate Birkhoff-James orthogonality of bounded linear operators defined between normed linear spaces $\mathbb{X}$ and $\mathbb{Y}.$ As an application of the results obtained, we characterize smoothness of a bounded linear operator $T$ under the condition that $\mathbb{K}(\mathbb{X},\mathbb{Y}),$ the space of compact linear operators is an $M-$ideal in $\mathbb{L}(\mathbb{X},\mathbb{Y}),$ the space of bounded linear operators.
\end{abstract}

\maketitle
\section{Introduction.} 
Birkhoff-James orthogonality is one of the most important notions of orthogonality defined in a normed linear space. It plays an important role in determining the geometry of Banach space like smoothness, strict convexity etc. Because of its importance, several mathematicians including Dragomir \cite{D}, Chmieli\'{n}ski \cite{C} have generalized the notion of Birkhoff-James orthogonality in a normed linear space. The purpose of this paper is to characterize Birkhoff-James orthogonality and approximate Birkhoff-James orthogonality in the space of bounded linear operators. Using these characterizations we study smoothness of bounded linear operators defined between normed linear spaces. Before proceeding further, we introduce relevant notations and terminologies to be used throughout the paper.

Let $\mathbb{X},\mathbb{Y}$ denote normed linear spaces. Throughout the paper we assume that the normed linear spaces are real and of dimension greater than or equal to $2.$ Let $B_{\mathbb{X}}$ and $S_{\mathbb{X}}$ denote the unit ball and unit sphere of $\mathbb{X}$ respectively, i.e., $B_{\mathbb{X}}=\{x\in \mathbb{X}:\|x\|\leq 1\}$ and $S_{\mathbb{X}}=\{x\in \mathbb{X}:\|x\|= 1\}.$ Let $\mathbb{X}^*$ denote the dual space of $\mathbb{X}$ and $Ext({B_{\mathbb{X}}})$ denote the set of all extreme points of the unit ball of $\mathbb{X}.$ Let $\mathbb{L}(\mathbb{X},\mathbb{Y})$ and $\mathbb{K}(\mathbb{X},\mathbb{Y})$ denote the space of all bounded and compact linear operators respectively from $\mathbb{X}$ to $\mathbb{Y}.$ For $T\in \mathbb{L}(\mathbb{X},\mathbb{Y}),$ let $M_T$ denote the norm attainment set of $T$, i.e., $M_T=\{x\in S_{\mathbb{X}}:\|Tx\|=\|T\|\}.$ A sequence $\{x_n\}\subseteq S_{\mathbb{X}}$ is said to be a norming sequence for a bounded linear operator $T$, if $\|Tx_n\|\to \|T\|.$

For $x,y\in \mathbb{X},$ $x$ is said to Birkhoff-James orthogonal \cite{B,J} to $y$, written as $x\perp_B y$ if $\|x+\lambda y\|\geq \|x\|$ for all $\lambda \in \mathbb{R}.$ In \cite{Sa}, the author introduced the notions of sets $x^+$ and $x^-$ as follows.\\
For $x,y\in \mathbb{X},$ we say that $y\in x^+$ if $\|x+\lambda y\|\geq \|x\|$ for all $\lambda \geq 0$ and $y\in x^-$ if $\|x+\lambda y\|\geq \|x\|$ for all $\lambda \leq 0.$\\
 Dragomir \cite{D} and Chmieli\'{n}ski \cite{C} defined approximate Birkhoff-James orthogonality in the following way.\\
For $x,y\in \mathbb{X}$ and $\epsilon \in [0,1),$ $x\perp_D^{\epsilon}y$ if $\|x+\lambda y\|\geq \sqrt{1-\epsilon^2}\|x\|$ for all  $\lambda \in \mathbb{R}.$ \\
Let $x,y\in \mathbb{X}$ and $\epsilon \in [0,1).$ Then $x\perp_B^{\epsilon}y$ if $\|x+\lambda y\|^2\geq \|x\|^2-2\epsilon\|x\|\|\lambda y\|$ for all  $\lambda \in \mathbb{R}$.

An element $x\in \mathbb{X}\setminus \{0\}$ is said to be a smooth point of $\mathbb{X}$ if $J(x)$ is singleton, where $J(x)=\{f\in S_{\mathbb{X}^*}:f(x)=\|x\|\}.$ We  further need the notions of $L-$ideal and $M-$ideal \cite{HWW} in a Banach space which are defined as follows.

A linear projection $P$ on $\mathbb{X}$ is called an $L-$projection  if $\|x\|=\|Px\|+\|x-Px\|$ for all $x\in \mathbb{X}.$ A closed subspace $J \subseteq \mathbb{X}$ is called an $L-$summand (or $L-$ideal) if it is the range of an $L-$projection.  A closed subspace $J \subseteq \mathbb{X}$ is called an  $M-$ideal if $J^0$ is an $L-$summand of $\mathbb{X}^*,$ where $J^0=\{f\in \mathbb{X}^*:f(x)=0 ~\forall ~x\in J\},$ is the annihilator of $J.$ Note that if $P$ is a contractive projection, then for $x\in range(P)$ and $y\in \ker(P),$ $x\perp_B y.$  \\

In this paper, we characterize approximate Birkhoff-James orthogonality $(\perp_D^{\epsilon})$ in the space of compact linear operators defined between arbitrary Banach spaces. In \cite[Th. 3.2]{PSM}, the authors characterized approximate Birkhoff-James orthogonality ($\perp_B^{\epsilon}$) in the space of compact linear operators when the domain space is reflexive. We obtain an analogous result for bounded linear operators under some restriction on the norm attainment set of the operator. In \cite[Th. 2.1]{SPM} and \cite[Th. 2.1]{PSG}, the authors characterized Birkhoff-James orthogonality of compact linear operators when the domain space is reflexive. Here we generalize these two results to obtain characterization of Birkhoff-James orthogonality ``$ T \bot_B A$'', when $\mathbb{K}(\mathbb{X},\mathbb{Y})$ is an $M-$ideal in $\mathbb{L}(\mathbb{X},\mathbb{Y})$ and $dist(T,\mathbb{K}(\mathbb{X},\mathbb{Y})<\|T\|.$  In \cite[Th. 1]{GY}, Grz\c a\' slewicz  and Younis characterized smooth points of $ \mathbb{L}(\ell^p,E),$  where $1<p<\infty,$ $E$ is a Banach space and $\mathbb{K}(\ell^p,E)$ is an $M-$ideal in  $\mathbb{L}(\ell^p,E).$  We prove that the necessary part of the characterization  holds for smoothness of arbitrary $T\in \mathbb{L}(\mathbb{X},\mathbb{Y}),$ where $\mathbb{X},\mathbb{Y}$ are arbitrary normed linear spaces.

\section{Main results.}
 We begin this section with a characterization of  approximate Birkhoff-James orthogonality ($\perp_D^{\epsilon}$) in terms of extreme linear functionals. In \cite[Th. 1.1, pp 170]{S}, Singer characterized Birkhoff-James orthogonality. He proved the following theorem.

\begin{theorem}\cite[Th. 1.1, pp 170]{S}\label{th.singer}
Let $\mathbb{X}$ be a normed linear space. Then for $x,y\in \mathbb{X},~x\perp_B y \Leftrightarrow \exists~ \phi,\psi \in Ext(B_{\mathbb{X}^*})$ and $\alpha \in [0,1]$ such that $\phi(x)=\psi(x)=\|x\|$ and $\alpha \phi(y)+(1-\alpha)\psi(y)=0.$
\end{theorem}
  The following theorem generalizes this result.

\begin{theorem}\label{th.epsilond}
	Let $\mathbb{X}$ be a normed linear space. Let $x,y\in \mathbb{X}$ and $\epsilon \in [0,1).$ Then $x\perp_D^{\epsilon} y$ if and only if there exists $\phi,\psi \in Ext(B_{\mathbb{X}^*}),~u\in span\{x,y\}\cap S_{\mathbb{X}}$ and $t\in [0,1]$ such that $\phi(u)=\psi(u)=1,$ $(1-t)\phi(x)+t\psi(x)\geq \sqrt{1-\epsilon^2}\|x\|$ and $(1-t)\phi(y)+t\psi(y)=0.$
\end{theorem}
\begin{proof}
	At first we prove the easier sufficient part of the theorem. Let $f=(1-t)\phi+t\psi.$ Then $f(u)=1,~f(y)=0$ and $\|f\|=1.$ Now, $f(x)=(1-t)\phi(x)+t\psi(x)\geq \sqrt{1-\epsilon^2}\|x\|.$  Therefore, by \cite[Lemma 3.2]{MSP}, we have, $x\perp_D^{\epsilon}y.$ \\
	Now, we prove the necessary part of the theorem. Let $x\perp_D^{\epsilon}y.$ Let $Z=span\{x,y\}.$ Then from \cite[Lemma 3.2]{MSP}, there exists $f\in S_{\mathbb{X}^*}$ such that  $f(y)=0$ and $f(x)\geq \sqrt{1-\epsilon^2}\|x\|.$ It is easy to see that  $f$ attains its norm in $Z,$ i.e.,  there exists $u\in S_Z$ such that $f(u)=1.$ Now, let $g=f|_{Z}.$  Thus, $g\in S_{Z^*}.$ Therefore, by \cite[Lemma 1.1, pp 166]{S}, there exists $g_1,g_2\in Ext(B_{Z^*})$ and $t\in [0,1]$ such that $g=(1-t)g_1+tg_2.$ Now, $g(u)=1\Rightarrow g_1(u)=g_2(u)=1.$ Again, by \cite [Lemma 1.2, pp 168]{S}, there exists $\phi,\psi\in Ext(B_{\mathbb{X}^*})$ such that $\phi|_{Z}=g_1$ and $\psi|_Z=g_2.$ Clearly, $\phi(u)=g_1(u)=1$ and $\psi(u)=g_2(u)=1.$  Now,
	\begin{eqnarray*}
		(1-t)\phi(x)+t\psi(x)&=&(1-t)g_1(x)+tg_2(x)\\
		&=& g(x)\\
		&=& f(x)\\
		& \geq&\sqrt{1-\epsilon^2}\|x\|~and\\
		(1-t)\phi(y)+t\psi(y)&=&(1-t)g_1(y)+tg_2(y)\\  
		&=&g(y)=0.
	\end{eqnarray*}
	This completes the proof of the theorem.
\end{proof}

In \cite[Th. 2.1]{PSM}, the authors characterized approximate Birkhoff-James orthogonality $(\perp_D^{\epsilon})$ of compact linear operators defined on a reflexive Banach space. Using Theorem \ref{th.epsilond}, we now obtain a characterization of approximate Birkhoff-James orthogonality $(\perp_D^{\epsilon})$ of compact linear operators defined between arbitrary Banach spaces, not necessarily reflexive. For this we need the following theorem. 
\begin{theorem}\cite[Th. 1.3]{RS}\label{th.rs}
Let $\mathbb{X},\mathbb{Y}$ be Banach spaces. Then 
\[Ext(B_{\mathbb{K}(\mathbb{X},\mathbb{Y})^*})=\{x^{**}\otimes y^{*}\in \mathbb{K}(\mathbb{X},\mathbb{Y})^{*}:x^{**}\in Ext(B_{\mathbb{X}^{**}}), y^*\in Ext(B_{\mathbb{Y}^*})\},\]
where $(x^{**}\otimes y^{*})(T)=x^{**}(T^*y^*)$ for every $T\in \mathbb{K}(\mathbb{X},\mathbb{Y}).$
\end{theorem}

\begin{theorem}\label{th.epsilondop}
	Let $\mathbb{X},\mathbb{Y}$ be Banach spaces. Let $T,A\in \mathbb{K}(\mathbb{X},\mathbb{Y})$ and $\epsilon\in[0,1).$ Then $T\perp_D^{\epsilon}A$ if and only if there exists $S\in span \{T,A\},$ $y_1^*,y_2^*\in Ext(B_{\mathbb{Y}^*}),$ $x_1^{**},x_2^{**}\in Ext(B_{\mathbb{X}^{**}})$ and $t\in [0,1]$ such that $\|S\|=1,$ $x_i^{**}(S^*y_i^*)=1$ for $i=1,2,$ $(1-t)x_1^{**}(T^*y_1^*)+tx_2^{**}(T^*y_2^*)\geq \sqrt{1-\epsilon^2}\|T\|$ and $(1-t)x_1^{**}(A^*y_1^*)+tx_2^{**}(A^*y_2^*)=0.$ 
\end{theorem}
\begin{proof}
	We first prove the necessary part of the theorem. Let $T\perp_D^{\epsilon}A.$ Then by Theorem \ref{th.epsilond}, there exists $S\in span \{T,A\},$ $\phi,\psi\in Ext(B_{\mathbb{K}(\mathbb{X},\mathbb{Y})^*})$ and $t\in [0,1]$ such that $\|S\|=1, \phi(S)=\psi(S)=1,$ $(1-t)\phi(T)+t\psi(T)\geq\sqrt{1-\epsilon^2}\|T\|$ and $(1-t)\phi(A)+t\psi(A)=0.$ By Theorem \ref{th.rs}, we obtain $\phi=x_1^{**}\otimes y_1^*$ and $\psi=x_2^{**}\otimes y_2^*,$ where $x_i^{**}\in Ext(B_{\mathbb{X}^{**}}),~y_i^*\in Ext(B_{\mathbb{Y}^*}),$ for $i=1,2.$ Therefore, $(1-t)(x_1^{**}\otimes y_1^*)(T)+t(x_2^{**}\otimes y_2^*)(T)\geq \sqrt{1-\epsilon^2}\|T\|$ and $(1-t)(x_1^{**}\otimes y_1^*)(A)+t(x_2^{**}\otimes y_2^*)(A)=0.$ Thus, $(1-t)x_1^{**}(T^* y_1^*)+tx_2^{**}(T^* y_2^*)\geq \sqrt{1-\epsilon^2}\|T\|$ and $(1-t)x_1^{**}(A^* y_1^*)+tx_2^{**}(A^* y_2^*)=0.$ Now, $\phi(S)=\psi(S)=1\Rightarrow (x_i^{**}\otimes y_i^*)(S)=1$ for $i=1,2,$ i.e., $x_i^{**}(S^* y_i^*)=1.$\\
	For the sufficient part, suppose that there exists $S\in span \{T,A\},$ $y_1^*,y_2^*\in Ext(B_{\mathbb{Y}^*}),$ $x_1^{**},x_2^{**}\in Ext(B_{\mathbb{X}^{**}})$ and $t\in [0,1]$ such that $\|S\|=1,$ $x_i^{**}(S^*y_i^*)=1$ for $i=1,2,$ $(1-t)x_1^{**}(T^*y_1^*)+tx_2^{**}(T^*y_2^*)\geq \sqrt{1-\epsilon^2}\|T\|$ and $(1-t)x_1^{**}(A^*y_1^*)+tx_2^{**}(A^*y_2^*)=0.$ Let $\phi_{i}=x_i^{**}\otimes y_i^*$ for $i=1,2.$ Then by Theorem \ref{th.rs}, $\phi_{i}\in Ext(B_{\mathbb{K}(\mathbb{X},\mathbb{Y})^*}),$ for $i=1,2.$ Now, $\phi_{i}(S)=(x_i^{**}\otimes y_i^*)(S)=1$ for $i=1,2,$ $(1-t)\phi_1(T)+t\phi_2(T)=(1-t)(x_1^{**}\otimes y_1^*)(T)+t(x_2^{**}\otimes y_2^*)(T)\geq \sqrt{1-\epsilon^2}\|T\|$ and $(1-t)\phi_1(A)+t\phi_2(A)=(1-t)(x_1^{**}\otimes y_1^*)(A)+t(x_2^{**}\otimes y_2^*)(A)=0.$  Thus, by Theorem \ref{th.epsilond}, $T\perp_D^{\epsilon}A.$
\end{proof}

In addition, if $\mathbb{X}$ is  reflexive then  following the same method as in Theorem \ref{th.epsilondop} and using the fact that $ Ext(B_{\mathbb{K}(\mathbb{X},\mathbb{Y})^*})=\{y^{*}\otimes x \in \mathbb{K}(\mathbb{X},\mathbb{Y})^{*}:y^{*}\in Ext(B_{\mathbb{Y}^{*}}), x\in Ext(B_{\mathbb{X}})\},$ where $ (y^* \otimes x) (T) = y^*(Tx),$ we can prove the following theorem.

\begin{theorem}\label{th.epsilondopr}
	Let $\mathbb{X}$ be a reflexive Banach space and $\mathbb{Y}$ be a Banach space. Let $T,A\in \mathbb{K}(\mathbb{X},\mathbb{Y})$ and $\epsilon\in[0,1).$ Then $T\perp_D^{\epsilon}A$ if and only if there exists $S\in span \{T,A\},$ $y_1^*,y_2^*\in Ext(B_{\mathbb{Y}^*}),~x_1,x_2\in Ext(B_{\mathbb{X}})$ and $t\in [0,1]$ such that $\|S\|=1,$ $y_i^{*}(Sx_i)=1$ for $i=1,2,$ $(1-t)y_1^{*}(Tx_1)+ty_2^{*}(Tx_2)\geq \sqrt{1-\epsilon^2}\|T\|$ and $(1-t)y_1^{*}(Ax_1)+ty_2^{*}(Ax_2)=0.$ 
\end{theorem}

In the following theorem, we obtain a characterization of approximate Birkhoff-James orthogonality ($\perp_B^{\epsilon}$) of bounded linear operators under some additional condition on the norm attainment set of the operator.
\begin{theorem}\label{th.lideal}
	Let $\mathbb{X},\mathbb{Y}$ be normed linear spaces. Let $T\in \mathbb{L}(\mathbb{X},\mathbb{Y})$ be such that $M_T=\{\pm x_0\},$ for some $x_0\in S_{\mathbb{X}}.$ Let span$\{x_0\}$ be an $L-$ideal of $\mathbb{X}$ and $P$ be the $L-$projection with range$(P)=$ span$\{x_0\}$. Suppose that $\|T\|_{ker(P)}<\|T\|.$ Then for any $A\in \mathbb{L}(\mathbb{X},\mathbb{Y})$ and $\epsilon\in [0,1),$ $T\perp_B^{\epsilon}A$ if and only if $\|Tx_0+\lambda Ax_0\|^2\geq \|Tx_0\|^2-2\epsilon\|Tx_0\|\|\lambda A\|$ for all $\lambda \in \mathbb{R}.$\\
	Moreover, if $x_0\in M_A,$ then $T\perp_B^{\epsilon}A$ if and only if $Tx_0\perp_B^{\epsilon}Ax_0.$
\end{theorem}

\begin{proof}
	The sufficient part of the proof is trivial. We only prove the necessary part. Let $T\perp_B^{\epsilon}A$ and $\|T\|_{ker(P)}=m<\|T\|.$ We claim that if $\{x_n\}$ is any norming sequence for $T,$ then $\{x_n\}$ has a subsequence converging to $ax_0,$ where $|a|=1.$ Clearly, for each $n\in \mathbb{N},$ $x_n=a_nx_0+y_n,$ where $a_n\in \mathbb{R},$ $y_n\in ker(P)$ and $1=\|x_n\|=|a_n|+\|y_n\|.$ Since $\{a_n\}$ is a bounded sequence of real numbers, it has a convergent subsequence $\{a_{n_k}\}$ converging to $a,$ say. Clearly, $|a|\leq 1.$ Now, $\|T(\frac{y_n}{\|y_n\|})\|\leq \|T\|_{ker(P)}=m\Rightarrow \|Ty_n\|\leq m\|y_n\|.$ Therefore,
	\begin{eqnarray*}
	\|Tx_{n_k}\| &=& \|a_{n_k}Tx_0+Ty_{n_k}\|\\
	         &\leq & |a_{n_k}|\|Tx_0\|+\|Ty_{n_k}\|\\
	         &\leq& |a_{n_k}|\|T\|+m\|y_{n_k}\|\\
	         &\leq &|a_{n_k}|\|T\|+m(1-|a_{n_k}|)\\
	 \Rightarrow \|Tx_{n_k}\|-m &\leq & |a_{n_k}|(\|T\|-m)\\
	 \Rightarrow \|T\|-m &\leq & |a| (\|T\|-m),~[taking~limit~k\to \infty]\\
	 \Rightarrow 1&\leq & |a|       
	\end{eqnarray*}
Thus, $a_{n_k}\to a,$ where $|a|=1.$ Hence, $y_{n_k}\to 0$ and $x_{n_k}\to ax_0,$ where $|a|=1.$ This proves our claim. Since $T\perp_B^{\epsilon}A,$ from \cite[Th. 3.4]{PSM}, we have, either $(a)$ or $(b)$ holds.\\
$(a)$ There exists a norming sequence $\{x_n\}$ for $T$ such that $\lim_{n\to \infty}\|Ax_n\|\leq \epsilon \|A\|.$ \\
$(b)$ There exists two norming sequences $\{x_n\},\{y_n\}$ for $T$  and two sequences of positive real numbers $\{\epsilon_n\},\{\delta_n\}$ such that\\
$(i)$ $\epsilon_n\to 0,~\delta_n\to 0$\\
$(ii)$ $\|Tx_n+\lambda Ax_n\|^2\geq(1-\epsilon_n^2)\|Tx_n\|^2-2\epsilon\sqrt{1-\epsilon_n^2}\|Tx_n\|\|\lambda A\|$ for all $\lambda \geq 0.$\\
$(iii)$ $\|Ty_n+\lambda Ay_n\|^2\geq(1-\delta_n^2)\|Ty_n\|^2-2\epsilon\sqrt{1-\delta_n^2}\|Ty_n\|\|\lambda A\|$ for all $\lambda \leq 0.$\\
First suppose that $(a)$ holds. Then there exists a subsequence $\{x_{n_k}\}$ of $\{x_n\}$ converging to $ax_0,$ where $|a|=1.$ Therefore, $\|Ax_0\|=\lim_{k\to \infty}\|Ax_{n_k}\|\leq \epsilon \|A\|.$ Thus, for all $\lambda \in \mathbb{R},$ $\|Tx_0+\lambda Ax_0\|^2\geq (\|Tx_0\|-|\lambda|\|Ax_0\|)^2\geq\|Tx_0\|^2-2\|Tx_0\||\lambda|\|Ax_0\|\geq \|Tx_0\|^2-2\epsilon\|Tx_0\|\|\lambda A\|.$\\
Now, suppose that $(b)$ holds. Then there exists subsequences $\{x_{n_k}\}$ and $\{y_{n_k}\}$ of $\{x_n\}$ and $\{y_n\}$ respectively, such that $x_{n_k}\to ax_0$ and $y_{n_k}\to bx_0,$ where $|a|=|b|=1.$ Now, in $(ii)$ and $(iii)$ of $(b),$ taking limit we obtain,  $\|Tx_0+\lambda Ax_0\|^2\geq \|Tx_0\|^2-2\epsilon\|Tx_0\|\|\lambda A\|$ for all $\lambda \in \mathbb{R}.$\\
Clearly, if $x_0\in M_A,$ then $T\perp_B^{\epsilon}A$ if and only if $Tx_0\perp_B^{\epsilon}Ax_0.$  This completes the proof of the theorem.
\end{proof}

\begin{remark}
	Note that, if $\mathbb{X}$ is an $L^1-$predual space, then for every $x^*\in Ext(B_{\mathbb{X}^*}),$ span$\{x^*\}$ is an $L-$ideal of $\mathbb{X}^*.$ Now, if $T\in \mathbb{L}(\mathbb{X},\mathbb{X})$ is such that $M_{T^*}=\{\pm x^*\}$ and $\|T^*\|_{ker(P)}<\|T^*\|,$ where $P$ is the $L-$projection with range$(P)=$ span$\{x^*\},$ then $T^*$ satisfies the hypothesis of Theorem \ref{th.lideal}.
\end{remark}

Using Theorem \ref{th.lideal}, we obtain the following characterization of Birkhoff-James orthogonality of bounded linear operators.

\begin{cor}\label{cor.lideal}
	Let $\mathbb{X},\mathbb{Y}$ be normed linear spaces. Let $T\in \mathbb{L}(\mathbb{X},\mathbb{Y})$ be such that $M_T=\{\pm x_0\},$ for some $x_0\in S_{\mathbb{X}}.$  Let span$\{x_0\}$ be an $L-$ideal of $\mathbb{X}$ and $P$ be the $L-$projection with range$(P)=$ span$\{x_0\}$. Suppose that $\|T\|_{ker(P)}<\|T\|.$ Then for any $A\in \mathbb{L}(\mathbb{X},\mathbb{Y}),$ $T\perp_B A$ if and only if $Tx_0\perp_B Ax_0.$
\end{cor}

In \cite[Th. 2.1]{SPM} and \cite[Th. 2.1]{PSG}, the authors characterized Birkhoff-James orthogonality of compact linear operators when the domain is reflexive. In the next theorem, we obtain analogous results for bounded linear operators under some restriction on the spaces.

\begin{theorem}\label{th.midealortho}
	(i) Let $\mathbb{X}$ be a reflexive Banach space and $\mathbb{Y}$ be a Banach space. Suppose that $\mathbb{K}(\mathbb{X},\mathbb{Y})$ is an $M-$ideal in $\mathbb{L}(\mathbb{X},\mathbb{Y}).$ Let $T\in \mathbb{L}(\mathbb{X},\mathbb{Y})$ be such that $\|T\|=1$ and $dist(T,\mathbb{K}(\mathbb{X},\mathbb{Y}))<1.$ Then for any $A\in \mathbb{L}(\mathbb{X},\mathbb{Y}),~T\perp_B A$ if and only if there exists $x_1,x_2\in M_T \cap Ext(B_{\mathbb{X}})$ such that $Ax_1\in (Tx_1)^+$ and $Ax_2\in (Tx_2)^-.$\\
	(ii) In addition, if $M_T=D\cup(-D),$ where $D$ is a closed connected subset of $S_{\mathbb{X}},$ then $T\perp_B A$ if and only if there exists $x\in M_T$ such that $Tx\perp_B Ax.$
\end{theorem}
\begin{proof}
    $(i)$ The sufficient part of the theorem is trivial. We only prove the necessary part of the theorem. Let $T\perp_B A.$ Then by Theorem \ref{th.singer}, there exists $\phi, \psi \in Ext(B_{{\mathbb{L}(\mathbb{X},\mathbb{Y})}^*})\cap J(T)$ and $\alpha \in [0,1]$ such that $\alpha \phi(A)+(1-\alpha)\psi(A)=0.$ Clearly, $\phi,\psi \in Ext(J(T)).$ Therefore, by \cite[Lemma 3.1]{W}, there exists $x_i \in  M_T \cap Ext(B_{\mathbb{X}})$ and $y_i^*\in Ext(J(Tx_i))$ for $i=1,2$ such that $\phi=y_1^*\otimes x_1$ and $\psi=y_2^*\otimes x_2.$ Now, 
    \begin{eqnarray*}
    \alpha \phi(A)+(1-\alpha)\psi(A) &=& 0\\
    \Rightarrow \alpha y_1^*\otimes x_1(A)+(1-\alpha)y_2^*\otimes x_2(A)&=&0\\
    \Rightarrow \alpha y_1^*(Ax_1)+(1-\alpha)y_2^*(Ax_2)&=&0
    \end{eqnarray*} 
	Therefore, without loss of generality we may assume that $y_1^*(Ax_1)\geq 0$ and $y_2^*(Ax_2)\leq 0.$ Thus, for any $\lambda \geq 0,$ $\|Tx_1+\lambda Ax_1\|\geq |y_1^*(Tx_1+\lambda Ax_1)|\geq \|Tx_1\|,$ since $y_1^*\in Ext(J(Tx_1)).$ Hence, $Ax_1\in (Tx_1)^+.$ Similarly, $y_2^*(Ax_2)\leq 0$ and $y_2^*\in Ext(J(Tx_2))$ gives that $Ax_2\in (Tx_2)^-.$ \\
	$(ii)$ The sufficient part of the theorem is trivial. We only prove the necessary part. Let $T\perp_B A$ and $M_T=D\cup(-D),$ where $D$ is a closed connected subset of $S_{\mathbb{X}}.$ Consider the sets 
	\[W_1=\{x\in D:Ax\in (Tx)^+\},\]
	\[W_2=\{x\in D:Ax\in (Tx)^-\}.\]
	Then by $(i),$ $W_1\neq \emptyset,~W_2\neq\emptyset.$ It is easy to check that $W_1$ and $W_2$ are closed. Applying \cite[Prop. 2.1 (i)]{Sa}, we obtain $D=W_1\cup W_2.$ Since $D$ is connected, $W_1\cap W_2\neq \emptyset.$ Let $x\in W_1\cap W_2.$ Then by \cite[Prop. 2.1 (ii)]{Sa}, $Tx\perp_B Ax.$ This completes the proof of the theorem.
\end{proof}

We next give a characterization of ``$T\perp_B^{\epsilon} A$'' under the condition that  $\mathbb{K}(\mathbb{X},\mathbb{Y})$ is an $M-$ideal in $\mathbb{L}(\mathbb{X},\mathbb{Y})$ and $dist(T,\mathbb{K}(\mathbb{X},\mathbb{Y}))<\|T\|.$

\begin{theorem}
	(i) Let $\mathbb{X}$ be a reflexive Banach space and $\mathbb{Y}$ be a Banach space. Suppose that $\mathbb{K}(\mathbb{X},\mathbb{Y})$ is an $M-$ideal in $\mathbb{L}(\mathbb{X},\mathbb{Y}).$ Let $T\in \mathbb{L}(\mathbb{X},\mathbb{Y})$ be such that $\|T\|=1$ and $dist(T,\mathbb{K}(\mathbb{X},\mathbb{Y}))<1.$ Then for any $A\in \mathbb{L}(\mathbb{X},\mathbb{Y})$ and $\epsilon\in[0,1),$ $T\perp_B^{\epsilon} A$ if and only if there exists $x_1,x_2\in M_T \cap Ext(B_{\mathbb{X}})$ such that 
	\[\|Tx_1+\lambda Ax_1\|^2\geq \|T\|^2-2\epsilon \|T\|\|\lambda A\|~\forall~\lambda~\geq~0~and\]
	\[\|Tx_2+\lambda Ax_2\|^2\geq \|T\|^2-2\epsilon \|T\|\|\lambda A\|~\forall~\lambda~\leq~0.\]
	(ii) In addition, if $M_T=D\cup(-D),$ where $D$ is a closed connected subset of $S_{\mathbb{X}},$ then $T\perp_B^{\epsilon} A$ if and only if there exists $x\in M_T$ such that $\|Tx+\lambda Ax\|^2\geq \|T\|^2-2\epsilon \|T\|\|\lambda A\|$ for all $\lambda\in \mathbb{R}.$
\end{theorem}
\begin{proof}
$(i)$ The sufficient part of the theorem is trivial. We only prove the necessary part of the theorem. Let $T\perp_B^{\epsilon} A.$ Then by \cite[Th. 2.2]{CSW}, there exists $S\in span\{T,A\}$ such that $T\perp_B S$ and $\|S-A\|\leq \epsilon \|A\|.$ Since $T\perp_B S,$ by Theorem \ref{th.midealortho}, there exists   $x_1,x_2\in M_T \cap Ext(B_{\mathbb{X}})$ such that $Sx_1\in (Tx_1)^+$ and $Sx_2\in (Tx_2)^-.$ Now, for $i=1,2,$ $\|Sx_i-Ax_i\|\leq \|S-A\|\leq \epsilon \|A\|.$ Suppose that $\lambda \geq 0.$ If $\|Tx_1\|-2\epsilon\|\lambda A\|<0,$ then $\|T\|^2-2\epsilon\|T\|\|\lambda A\|=\|Tx_1\|^2-2\epsilon\|Tx_1\|\|\lambda A\|<0\leq\|Tx_1+\lambda Ax_1\|^2.$ If $\|Tx_1\|-2\epsilon\|\lambda A\|\geq 0,$ then 
\begin{eqnarray*}
\|Tx_1+\lambda Ax_1\|^2 &=& \|Tx_1+\lambda Sx_1-\lambda Sx_1+\lambda Ax_1\|^2\\
&\geq & (\|Tx_1+\lambda Sx_1\|-|\lambda|\| Sx_1- Ax_1\|)^2\\
&\geq & \|Tx_1+\lambda Sx_1\|^2-2\|Tx_1+\lambda Sx_1\||\lambda|\| Sx_1- Ax_1\|\\
&\geq & \|Tx_1+\lambda Sx_1\|^2-2\|Tx_1+\lambda Sx_1\||\lambda|\epsilon\| A\|\\
&\geq &  \|Tx_1+\lambda Sx_1\| ( \|Tx_1+\lambda Sx_1\|-2\epsilon\|\lambda A\|)\\
&\geq& \|Tx_1\|(\|Tx_1\|-2\epsilon\|\lambda A\|)~[Since~Sx_1\in (Tx_1)^+]\\
&=&\|T\|^2-2\epsilon\|T\|\|\lambda A\|.
\end{eqnarray*} 
Thus, we get, $\|Tx_1+\lambda Ax_1\|^2\geq \|T\|^2-2\epsilon \|T\|\|\lambda A\|$ for all $\lambda \geq 0.$
Similarly,  using $Sx_2\in (Tx_2)^-,$ it can be proved that $\|Tx_2+\lambda Ax_2\|^2\geq \|T\|^2-2\epsilon \|T\|\|\lambda A\|$ for all $\lambda \leq 0.$ \\
$(ii)$ The proof follows easily from Theorem \ref{th.midealortho} and the method adopted in part $(i)$ of this theorem.
\end{proof}

\section{Applications.}
As an application of results obtained in the previous section we study smoothness of bounded linear operators defined between arbitrary normed linear spaces. Recently, many authors \cite{PSG,SPMR,R,Ra} have studied smoothness of bounded linear operators. We first obtain a characterization of smoothness of bounded linear operator $T$  when $M_T=\{\pm x_0\}$ and span$\{x_0\}$ is an $L-$ideal of $\mathbb{X}.$

\begin{theorem}
	Let $\mathbb{X},\mathbb{Y}$ be normed linear spaces. Let $T\in \mathbb{L}(\mathbb{X},\mathbb{Y})$ be such that $M_T=\{\pm x_0\},$ for some $x_0\in S_{\mathbb{X}}.$  Let span$\{x_0\}$ be an $L-$ideal of $\mathbb{X}.$ Then $T$ is smooth if and only if the following conditions hold.\\
	(i) $Tx_0$ is a smooth point in $\mathbb{Y}.$\\
	(ii) $\|T\|_{ker(P)}<\|T\|,$ where $P$ is the $L-$projection with range$(P)=$ span$\{x_0\}$.
\end{theorem}
\begin{proof}
The proof follows easily from Corollary \ref{cor.lideal}, \cite[Th. 3.3]{SPMR}, \cite[Th. 4.4]{PSG} and the fact that $x_0\perp_B ker(P).$
\end{proof}
In \cite{GY}, Grz\c a\' slewicz and Younis characterized smooth points of $\mathbb{L}(\ell^p,E),$ where $1<p<\infty,$ $E$ is a Banach space and $\mathbb{K}(\ell^p,E)$ is an M-ideal in $\mathbb{L}(\ell^p,E).$ In the following theorem, we generalize the necessary part of the Theorem \cite[Th. 1]{GY}, the proof of which follows easily from \cite[Th. 3.3]{SPMR}.

\begin{theorem}\label{th.smoothnecessary}
	Let $\mathbb{X},\mathbb{Y}$ be normed linear spaces. Let $T\in \mathbb{L}(\mathbb{X},\mathbb{Y})$ be such that $M_T \neq \emptyset.$ Suppose that $T$ is smooth. Then the following conditions hold.\\
	(i) $M_T=\{\pm x_0\}$ for some $x_0\in S_{\mathbb{X}}.$\\
	(ii) $Tx_0$ is smooth point in $\mathbb{Y}$.\\
	(iii) $dist(T,\mathbb{K}(\mathbb{X},\mathbb{Y}))<\|T\|.$ 
\end{theorem}

Grz\c a\' slewicz and Younis \cite[Lemma 1]{GY} proved that the converse of Theorem \ref{th.smoothnecessary} is also true if we additionally assume that $\mathbb{X}$ is a reflexive Banach space and  $\mathbb{K}(\mathbb{X},\mathbb{Y})$ is an $M-$ideal in $\mathbb{L}(\mathbb{X},\mathbb{Y}).$ Using Theorem \ref{th.midealortho}, we give an alternative proof of \cite[Lemma 1]{GY}.

\begin{theorem}\label{th.smoothsufficient}
	Let $\mathbb{X}$ be a reflexive Banach space and $\mathbb{Y}$ be a Banach space. Let $\mathbb{K}(\mathbb{X},\mathbb{Y})$ be an $M-$ideal in $\mathbb{L}(\mathbb{X},\mathbb{Y}).$ Let $T\in \mathbb{L}(\mathbb{X},\mathbb{Y})$ be such that the following conditions hold.\\
	(i) $M_T=\{\pm x_0\}$ for some $x_0\in S_{\mathbb{X}}.$\\
	(ii) $Tx_0$ is smooth point in $\mathbb{Y}$.\\
	(iii) $dist(T,\mathbb{K}(\mathbb{X},\mathbb{Y}))<\|T\|.$ \\
	Then $T$ is smooth.
\end{theorem}
\begin{proof}
	Without loss of generality assume that $\|T\|=1.$ Let $A,B\in \mathbb{L}(\mathbb{X},\mathbb{Y})$ be such that $T\perp_B A$ and $T\perp_B B.$ Then by Theorem \ref{th.midealortho}, $Tx_0\perp_B Ax_0$ and $Tx_0\perp_B Bx_0.$ Since $Tx_0$ is smooth, by \cite[Th. 4.2]{J}, $Tx_0\perp_B(Ax_0+Bx_0).$ Thus, $T\perp_B (A+B),$ since $x_0\in M_T.$ Hence, $T$ is smooth. This completes the proof of the theorem.
\end{proof}

Combining Theorem \ref{th.smoothnecessary} and Theorem \ref{th.smoothsufficient}, we obtain the following theorem, characterizing the smooth points of $\mathbb{L}(\mathbb{X},\mathbb{Y}),$ where $\mathbb{X}$ is a reflexive Banach space, $\mathbb{Y}$ is a Banach space and $\mathbb{K}(\mathbb{X},\mathbb{Y})$ is an $M-$ideal in $\mathbb{L}(\mathbb{X},\mathbb{Y}).$

\begin{theorem}
	Let $\mathbb{X}$ be a reflexive Banach space and $\mathbb{Y}$ be a Banach space. Let $\mathbb{K}(\mathbb{X},\mathbb{Y})$ be an $M-$ideal in $\mathbb{L}(\mathbb{X},\mathbb{Y}).$ Let $T\in \mathbb{L}(\mathbb{X},\mathbb{Y})$ be such that $M_T\neq \emptyset.$ Then $T$ is smooth if and only if the following conditions hold.\\
	(i) $M_T=\{\pm x_0\}$ for some $x_0\in S_{\mathbb{X}}.$\\
	(ii) $Tx_0$ is smooth point in $\mathbb{Y}$.\\
	(iii) $dist(T,\mathbb{K}(\mathbb{X},\mathbb{Y}))<\|T\|.$ 
\end{theorem}

\begin{remark}
We would like to remark that if $\mathbb{X}^*$ and $\mathbb{Y}$ are separable then $\mathbb{K}(\mathbb{X},\mathbb{Y})$ is also separable. A well-known theorem of Mazur asserts that in a separable Banach space smooth points are dense (see \cite[pp.171]{H}). Recently Martin \cite{M} proved the existence of compact operators that cannot be approximated by norm attaining operators. From these facts it is easy to see that there is a compact operator which is smooth but does not attain its norm. Thus for smoothness of  $T\in \mathbb{L}(\mathbb{X},\mathbb{Y})$ it is not necessary that $ M_T \neq \emptyset.$  However, assuming that $ M_T \neq \emptyset,$ we can raise the following open question.
\end{remark}

\noindent \textbf{Question. } Let $\mathbb{X},\mathbb{Y}$ be normed linear spaces and  $T\in \mathbb{L}(\mathbb{X},\mathbb{Y})$ be such that \\
  (i) $M_T=\{\pm x_0\}$ for some $x_0\in S_{\mathbb{X}}.$\\
	(ii) $Tx_0$ is smooth point in $\mathbb{Y}$.\\
	(iii) $dist(T,\mathbb{K}(\mathbb{X},\mathbb{Y}))<\|T\|.$ \\
		Then whether $T$ is smooth or not.

\end{document}